\documentclass[11pt]{amsart}
\usepackage{mathrsfs} 
\usepackage[margin=1.25in]{geometry}
\usepackage{ifthen}
\usepackage{graphicx}
\usepackage{epsfig}
\usepackage{dsfont, amssymb}
\usepackage{amsfonts,dsfont,amssymb,amsthm,stmaryrd,bbm}
\usepackage{amsmath}

\usepackage{xcolor}

\usepackage[toc,page]{appendix} 
\usepackage{hyperref,mathtools}
\hypersetup{colorlinks=true,pdftitle=""
}

\let\originallesssim\lesssim
\let\originalgtrsim\gtrsim

\DeclareRobustCommand{\lesssim}{%
  \mathrel{\mathpalette\lowersim\originallesssim}%
}
\DeclareRobustCommand{\gtrsim}{%
  \mathrel{\mathpalette\lowersim\originalgtrsim}%
}

\makeatletter
\newcommand{\lowersim}[2]{%
  \sbox\z@{$#1<$}%
  \raisebox{-\dimexpr\height-\ht\z@}{$\m@th#1#2$}%
}
\makeatother

\newtheorem{thm}{Theorem}

\newtheorem{remark}[thm]{Remark}
\newtheorem{lem}[thm]{Lemma}

\newtheorem{cor}[thm]{Corollary}

\newcommand\independent{\protect\mathpalette{\protect\independent}{\perp}} 
\def\independent#1#2{\mathrel{\rlap{$#1#2$}\mkern2mu{#1#2}}}

\DeclareMathOperator{\Var}{Var}

\def\Var{{\rm Var}}

\def\phi{\varphi}

\def\bee{\begin{eqnarray*}}
\def\ene{\end{eqnarray*}}
\usepackage{mathtools}

\begin{document}

\title[Quantitative cube slicing and min-entropy power inequality]{Quantitative form of Ball's Cube slicing in $\mathbb{R}^n$ and equality cases in the min-entropy power inequality}

\author{James Melbourne and Cyril Roberto}

\thanks{The last author is supported by the Labex MME-DII funded by ANR, reference ANR-11-LBX-0023-01 and ANR-15-CE40-0020-03 - LSD - Large Stochastic Dynamics, and the grant of the Simone and Cino Del Duca Foundation, France.}

\address{Centro de Investigaci\'on en Matem\'aticas, Probabilidad y Estad\'isticas.: 36023 Guanajuato, Gto, Mexico.}

\address{Universit\'e Paris Nanterre, Modal'X, FP2M, CNRS FR 2036, 200 avenue de la R\'epublique 92000 Nanterre, France.}

\email{james.melbourne@cimat.mx , croberto@math.cnrs.fr}

\date{\today}

\maketitle
 
\begin{abstract}
We prove a quantitative form of the celebrated Ball's theorem on cube slicing in $\mathbb{R}^n$ and obtain, as a consequence, equality cases in the min-entropy power inequality.
Independently, we also give a quantitative form of Khintchine's inequality in the special case $p=1$.
\end{abstract}

\section{Introduction}

In his seminal paper \cite{ball}, Keith Ball proved that the maximal $(n-1)$-dimensional volume of the section of the cube $C_n \coloneqq [-\frac{1}{2},\frac{1}{2}]^n$ by an hyperplane is $\sqrt{2}$. Therefore proving a conjecture by Hensley \cite{hensley}.

More precisely, for $a=(a_1,\dots,a_n) \in \mathbb{R}^n$ with $|a|:=\sqrt{a_1^2+\dots+a_n^2}=1$,
put $\sigma(a,t)= |C_n \cap H_{a,t}|_{n-1}$ for the volume of the intersection of the cube with the hyperplane  $H_{a,t}=\{x \in \mathbb{R}^n : \langle x,a\rangle =t\}$, where $\langle \cdot,\cdot\rangle $ is the usual scalar product in  $\mathbb{R}^n$ and $|\cdot|_{n-1}$ stands for the ($(n-1)$-dimensional) volume.

\begin{thm}[Ball \cite{ball}] \label{th:ball}
For all unit vector $a$ and all $t \in \mathbb{R}$, it holds  $\sigma(a,t) \leq \sqrt{2}$. 
Moreover, equality holds only if $t=0$ and $a$ has only two non-zero coordinates having value $\frac{1}{\sqrt{2}}$.
\end{thm}

Ball's result  means that the maximal volume of the sections of the cube by hyperplanes are achieved when the section is a product of a $(n-2)$-dimensional cube  $C_{n-2}$ with the  diagonal of a $2$-dimensional cube $C_2$. The original proof is based on Fourier tansform and series expansion. Alternative proofs can be found in \cite{NP00} (based on distribution functions) and very recently in \cite{MR} (by mean of a  transport argument). 

Ball used Theorem \ref{th:ball} to give a negative answer to the famous Busemann-Petty problem in dimension 10 and higher \cite{ball2}. His paper has inspired many research in convex geometry and is still very current. 
We refer to \cite{IT,K21,KR,KK,CKT} to quote just a few of the most recent papers in the field and refer to the reference therein for a more detailed description of the literature.

Our first main result is the following quantitative version of Ball's theorem.
\begin{thm} \label{thm: quantitative Ball}
Fix $\varepsilon \in (0, \frac{1}{75})$. Let $a \in \mathbb{R}^n$ with $|a|=1$ and $t \in \mathbb{R}$ be such that
$\sigma(a,t) \geq (1-\varepsilon)\sqrt{2}$. Then, there exists two indices $j_o, j_1$ such that
$$
\frac{1}{\sqrt{2}}(1-37.5  \varepsilon) \leq |a_{j_o}| , |a_{j_1}|  \leq \frac{1}{\sqrt{2}}(1 +2\varepsilon) .
$$
Moreover, $\sum_{j \neq j_o,j_1}a_j^2 \leq 50 \varepsilon$ and
in particular, for all $j \neq j_o, j_1$, 
$|a_j| \leq \sqrt{50 \varepsilon}$.
\end{thm}



Ball's slicing theorem, combined with a result of Rogozin \cite{rogozin1987estimate}, was used by Bobkov and Chistyakov \cite{bobkov2014bounds} to derive an optimal inequality for min-entropy power. Namely, they proved that
\begin{equation} \label{eq:bc}
    N_\infty(X_1+\dots+X_n) \geq \frac{1}{2} \sum_{i=1}^\infty N_\infty (X_i)
\end{equation}
for any independent random variables $X_1,\dots,X_n$, with $N_\infty$ the min-entropy power we now define. We may call the latter \emph{Bobkov-Chistyakov’s min-entropy power inequality}.

For a ($\mathbb{R}$ valued) random variable $X$, the \emph{min-Entropy power} is defined as
\begin{align*}
    N_\infty(X) = M^{-2}(X).
\end{align*}
when
\begin{align*}
    M(X) \coloneqq \inf \left\{ c : \mathbb{P}(X \in A) \leq c \ |A| \hbox{ for all Borel } A \right\} < \infty
\end{align*}
and $N_\infty(X) = 0$ otherwise.
When $X$ is absolutely continuous with respect to the Lebesgue measure, with density $f$, then $M(X) = \|f\|_\infty$ is the essential supremum of $f$ with respect to the Lebesgue measure.

The nomenclature ``min-entropy power'' is information theoretic. In that field  the entropy power inequality refers to the fundamental inequality due to Shannon \cite{shannon1948mathematical} which demonstrates that $X_i$ independent random variables with densities $f_i$ satisfy
\begin{align*}
    N(X_1 + \cdots + X_n) \geq  \sum_i N(X_i),
\end{align*}
where $N(X) = e^{2 h(X)}$ denotes the ``entropy power'', with the Shannon entropy $h(X) = - \int f(x) \log f(x) dx$.   The R\'enyi entropy \cite{renyi1961measures}, for $\alpha \in [0,\infty]$ defined as $h_\alpha(X) = \frac{ \int f^\alpha (x) dx}{1-\alpha}$ for $\alpha \in (0,1) \cup (1,\infty)$ and through continuous limits otherwise, gives a parameterized family of entropies that includes the usual Shannon entropy as a special case (by taking $\alpha=1$).  It can be easily seen (through Jensen's inequality, and the expression $h_\alpha(X) = \left( \mathbb{E} f^{\alpha-1}(X) \right)^{\frac 1{1- \alpha}}$) that for a fixed variable $X$, the R\'enyi entropy is decreasing in $\alpha$.  Thus for a fixed variable $X$, the parameter $\alpha = \infty$, $h_\infty(X) = - \log \|f \|_\infty$, furnishes the minimizer of the family $\{ h_\alpha(X) \}_\alpha$, and is often referred to as the ``min-entropy''.   Hence the terminology and notation min-entropy power used $N_\infty(X) = e^{2 h_\infty(X)}$ is in analogy with the Shannon entropy power $N(X) = e^{2 h(X)}$.  Entropy power inequalities for the full class of R\'enyi entropies have been a topic of recent interest in information theory, see \textit{e.g.}\ \cite{bobkov2015entropy, bobkov2017variants, li2018renyi, li2020further, marsiglietti2018entropy, ram2016renyi, rioul2018renyi}, and for more background we refer to \cite{madiman2017forward} and references therein. 

In \cite{bobkov2014bounds} it was observed in a closing remark that the constant $\frac 1 2$ in \eqref{eq:bc} is sharp.  Indeed by taking $n =2$ and $X_1$ and $X_2$ to be \textit{i.i.d.}\ uniform on an interval \eqref{eq:bc} is seen to hold with equality.   In the following theorem, we demonstrate that this is (essentially) the only equality case.  In fact, thanks to the quantitative form of Ball's slicing theorem above, we can derive a quantitative form of Bobkov-Chistyakov's min entropy power inequality, see Corollary \ref{Ball to min-EPI} below, that, in turn, allows us to characterize equality cases in \eqref{eq:bc} which constitutes our second main theorem.

\begin{thm}\label{thm:equality_cases}
For $X_1, \dots, X_n$ independent random variables,
\begin{align} \label{eq: min epi}
    N_\infty(X_1 + \cdots + X_n) \geq \frac 1 2 \sum_{i=1}^n N_\infty(X_i)
\end{align}
with equality if and only if there exists $i_1$ and $i_2$ and $x \in \mathbb{R}$ such that $X_{i_1}$ is uniform on a set $A$, and $X_{i_2}$ is a uniform distribution on $x-A$ and for $i \neq i_1, i_2$, $X_i$ is a point mass.
\end{thm}
Note that this is distinct from the $d$-dimensional case, see \cite{madiman2017rogozin}, where sharp constants can be approached asymptotically for $X_i$ \textit{i.i.d.}\ and uniform on a $d$-dimensional ball.  More explicitly, for $d \geq 2$, if $\Lambda$ denotes all finite collections of independent $\mathbb{R}^d$-valued random variables
\begin{align*}
    \sup_{X \in \Lambda} \frac{N_\infty(X_1 + \cdots + X_m)}{\sum_{i=1}^m N_\infty(X_i)} = \lim_{n \to \infty} \frac{N_\infty(Z_1 + \cdots + Z_n)}{\sum_{i=1}^n N_\infty(Z_i)},
\end{align*}
where $Z_i$ are \textit{i.i.d.}\ and uniform on a $d$-dimensional Euclidean unit ball.

\medskip 

We end with a quantitative Khintchine's inequality.
Though our result is independent, we stress that, as it is well known in the field and as it was pointed out by Ball himself in \cite[Additional remarks]{ball}, the inequality $\sigma(a,t) \leq \sqrt{2}$ of Theorem \ref{th:ball} is however related to Khintchine's inequalities.

Denote by $B_1, B_2, \dots$ symmetric $-1,1$-Bernoulli variables. Khintchine's inequalities assert that, for any $p \in (0,\infty)$ there exist some constant $A_p$, $A'_p$ such that for all $n$ and all $a=(a_1,\dots,a_n) \in \mathbb{R}^n$ it holds
\begin{equation} \label{eq:khintchine}
    A_p \left( \sum_{i=1}^n a_i^2 \right)^\frac{p}{2} 
    \leq 
    R_p(a):= \mathbb{E}\left[ \left|\sum_{i=1}^n a_i B_i\right|^p \right]
    \leq 
    A'_p \left( \sum_{i=1}^n a_i^2 \right)^\frac{p}{2} .
\end{equation}
Such inequalities were proved by Khintchine in a special case \cite{khintchine}, and studied in a more systematic way by Littlewood, Paley and Zygmund.

The best constants in \eqref{eq:khintchine} are known. 
This is due to Haagerup \cite{haagerup}, after partial results by
Steckin \cite{steckin}, Young \cite{young} and Szarek \cite{szarek}. In particular, Szarek proved that $A_1=1/\sqrt{2}$, that was a long outstanding conjecture of Littlewood, see \cite{hall}.

The connection between Theorem \ref{th:ball} and Khintchine's inequalities goes as follows: as fully derived in \cite{CKT}, Ball's theorem  can be rephrased as
$$
\mathbb{E} \left[ \left|\sum_{i=1}^n a_i \xi_i\right|^{-1} \right]
\leq \sqrt{2} \left( \sum_{i=1}^n a_i^2 \right)^{-\frac{1}{2}} 
$$
where $\xi_i$ are i.i.d. random vectors in $\mathbb{R}^3$ uniform on the centered Euclidean unit sphere $S^2$. As a result Ball's slicing of the cube can be seen as a sharp $L_{-1}-L_{2}$ Khintchine-type ienquality.

Our last main result is a quantitative version of (the lower bound in) Khintchine's inequality for $p=1$, that has the same flavour of Theorem \ref{thm: quantitative Ball} (thought being independent).

\begin{thm} \label{th:quantitative khintchine}
Fix $\varepsilon \in (0,1/100)$, an integer $n$ and $a=(a_1,\dots,a_n) \in \mathbb{R}^n$ such that $|a| = 1$, satisfying
$$
R_1(a) \leq \frac{1+\varepsilon}{\sqrt 2} .
$$
Then, there exists two indices $i_1,i_2$ such that 
$$
\frac{1- 30 \varepsilon}{\sqrt 2}\leq |a_{i_1}|, |a_{i_2}| \leq \frac{1+\varepsilon}{\sqrt 2}
$$
Also, it holds $\sum a_i^2 \leq 57 \varepsilon$ and in particular, for any $i \neq i_1, i_2$,
$|a_i| \leq \sqrt{57 \varepsilon}$.
\end{thm}

The proofs of Theorem \ref{thm: quantitative Ball} and Theorem \ref{th:quantitative khintchine} are based on a careful analysis of Ball's integral inequality
$$
\int_{-\infty}^\infty \left| \frac{\sin (\pi u)}{\pi u} \right|^s du \leq \sqrt{\frac{2}{s}},
\qquad s \geq 2
$$
and, respectively, Haagerup's integral inequality
$$
\int_0^\infty \left(1 - \left|\cos \left( \frac{u}{\sqrt{s}} \right)\right|^s\right) \frac{du}{u^{p+1}} \geq \int_{-\infty}^\infty \left(1-e^{-u^2/2} \right) \frac{du}{u^{p+1}}, \qquad s \geq 2 
$$
in the special case $p=1$. It is worth mentioning that Theorem \ref{th:quantitative khintchine} is restricted to $p=1$ because the latter integrals can be made explicit only in that case. In order to deal with general $p$ (at least $p \in [p_o,2)$, say, with $p_o \simeq 1.85$ implicitly defined through the Gamma function, see \cite{haagerup}),
one would need to study very carefully the map 
$F_p \colon s \mapsto \int_0^\infty \left(1 - \left|\cos \left( \frac{u}{\sqrt{s}} \right)\right|^s\right) \frac{du}{u^{p+1}}$ and prove that it is increasing and then decreasing on $[2,\infty)$ with careful control of its variations. The difficulty is also coming from the fact that, at $p=p_o$, $F_p(2)=F_p(\infty)$. This in particular makes the quantitative version difficult to state properly. Indeed, for $0<p<p_o$, the extremizers in the lower bound of \eqref{eq:khintchine} are those $a$ with two indices equal to $1/\sqrt{2}$ and the others vanishing. While for $p > p_o$, there are no extremizers for finite $n$ (the "extremizer" is $a=(\frac{1}{\sqrt n},\dots,\frac{1}{\sqrt n})$ in the limit (by the central limit theorem)). At $p=p_o$ the two "extremizers" coexist.
Theorem \ref{th:quantitative khintchine} is therefore only a first attempt in the understanding of quantitative forms of Khintchine's inequalities.  

\medskip

The next sections are devoted to the proof of Theorem \ref{thm: quantitative Ball}, Theorem \ref{thm:equality_cases} and Theorem \ref{th:quantitative khintchine}.

\section{Quantitative slicing: Proof of Theorem \ref{thm: quantitative Ball}}

In this section, we give a proof of Theorem \ref{thm: quantitative Ball}.
We need first to recall part of the original proof by Ball, based on Fourier and anti-Fourier transform. We may omit some details that can be found in 
\cite{ball}.

By symmetry we can assume without loss of generality that $a_j \geq 0$ for all $j$. Reducing the dimension of the problem if necessary, we will further reduces to $a_j \neq 0$ for all $j$.

In \cite{ball} it is proved that $\sigma(a,t) \leq \frac{1}{a_j}$ for all $j$ (see also \cite[step 1]{NP00}). The argument is geometric. Put $e_j:=(0,\dots,0,1,0,\dots0)$ for the $j$-th unit vector of the canonical basis. Then it is enough to observe that the volume of $C_n \cap H_{a,t}$ equals the volume of its projection to the hyperplane $H_{e_j,0}$ (orthogonal to the $j$-th direction) divided by the cosine of the angle of $a$ and $e_j$, that is precisely $a_j$, while the projection of $C_n$ on $H_{e_j,0}$ has volume $1$. Therefore $a_j \leq \frac{1}{\sqrt{2}(1-\varepsilon)} \leq \frac{1}{\sqrt{2}}(1 + 2 \varepsilon)$ for all $j$, which proves one inequality of Theorem \ref{thm: quantitative Ball}.

We follow the presentation of \cite[step 2]{NP00}. Let $\hat S$
be the Fourier transform of $S \colon t \mapsto \sigma(a,t)$. By definition, we have
\begin{align*}
\hat S(u) 
& = 
\int_\mathbb{R} \sigma(a,t)e^{-2i\pi u t} dt \\
& =
\int_{C_n} e^{-2 i \pi u \langle x,a \rangle} dx \\
& = 
\prod_{j=1}^n \int_{-\frac{1}{2}}^\frac{1}{2} e^{-2 i \pi u a_j a_j} dx_j \\
& =
\prod_{j=1}^n \frac{\sin(\pi a_j u)}{\pi a_j u} .
\end{align*}
Therefore, taking the  anti-Fourier transform, Ball obtained the following explicit formula\footnote{An alternative explicit formula is given by Franck and Riede \cite{FR} (with different normalization). 
The authors ask if there could be an alternative proof of Ball's theorem based on their formula.
} for $\sigma(a,t)$:
\begin{align*}
\sigma(a,t) 
& =
\int_{-\infty}^\infty \hat S(u) e^{2 i \pi u t} du \\
& = \int_{-\infty}^\infty e^{2\pi i ut} \prod_{j=1}^n \frac{\sin(\pi a_ju)}{\pi a_ju} du  .
\end{align*}
Applying Holder's inequality, since $a_1^2+\dots+a_n^2=1$, one gets
\begin{align} \label{eq:ball-holder}
\sigma(a,t)  
& \leq 
\int_{-\infty}^\infty  \prod_{j=1}^n \left| \frac{\sin(\pi a_ju)}{\pi a_ju} \right| du \nonumber \\
& \leq 
 \prod_{j=1}^n  \left( \int_{-\infty}^\infty\left| \frac{\sin(\pi a_ju)}{\pi a_ju} \right|^{1/a_j^2} du \right)^{a_j^2} .
\end{align}
Ball's theorem follows from the fact that 
$I(a_j) \coloneqq \int_{-\infty}^\infty\left| \frac{\sin(\pi a_ju)}{\pi a_ju} \right|^{1/a_j^2} du \leq \sqrt{2}$
with equality only if $a_j=1/\sqrt{2}$.
Changing variable, this is equivalent to proving that
\begin{equation} \label{eq:ball}
\int_{-\infty}^\infty\left| \frac{\sin(\pi u)}{\pi u} \right|^{s} du < \sqrt{\frac{2}{s}}
\end{equation}
for every $s >2$ (for $s=2$ this is an identity). The latter is known as Ball's integral inequality
and was proved in \cite{ball}\footnote{An asymptotic study of such integrals can be found in \cite{KOS}.} (see \cite{NP00,MR} for alternative approaches).

One key ingredient in the proof of Theorem \ref{thm: quantitative Ball} is a reverse form of Ball's integral inequality given in Lemma \ref{lem:s} below.

Turning to our quantitative question, observe that if for all $j=1,\dots,n$, $I(a_j) < (1-\varepsilon)\sqrt{2}$, then \eqref{eq:ball-holder} would imply that $\sigma(a,t) < (1-\varepsilon)\sqrt{2}$, a contradiction. Therefore, there must exist $j_o$ such that $I(a_{j_o}) \geq (1-\varepsilon)\sqrt{2}$.
The aim is now to prove that $a_{j_o}$ is closed to $1/\sqrt{2}$. In fact, changing variables ($s=1/a_{j_o}^2 \geq 2(1-\varepsilon)$), we observe that 
\begin{align*}
I(a_{j_o}) 
& =  
\int_{-\infty}^\infty\left| \frac{\sin(\pi a_ju)}{\pi a_ju} \right|^{1/a_j^2} du \\
& =
\sqrt{s} \int_{-\infty}^\infty\left| \frac{\sin(\pi u)}{\pi u} \right|^{s} du .
\end{align*}
Hence, $I(a_{j_o}) \geq (1-\varepsilon)\sqrt{2}$ is equivalent to saying that
$$
\int_{-\infty}^\infty\left| \frac{\sin(\pi u)}{\pi u} \right|^{s} du 
\geq  
(1-\varepsilon) \sqrt{\frac{2}{s}} .
$$
Lemma \ref{lem:s} guarantees that, if $s \geq 2$, then $s=\frac{1}{a_{j_o}^2} \leq 2 + 50 \varepsilon$.
If $s \leq 2$ then $\frac{1}{a_{j_o}^2} \leq 2$ which amounts to $a_{j_o} \geq \frac{1}{\sqrt{2}}$. In any case
\begin{align*}
a_{j_o} 
& \geq \frac{1}{\sqrt{2+50 \varepsilon}} \\
& \geq 
\frac{1}{\sqrt{2}} (1 - \frac{25}{2} \varepsilon)
\end{align*}
since $\frac{1}{\sqrt{1+t}} \geq 1-\frac{1}{2}t$ for any $t \in (0,1)$.

Iterating the argument, assume that for all $j \neq j_o$, $I(a_j) < (1-3\varepsilon)\sqrt{2}$. Since $I(a_{j_o}) \leq \sqrt{2}$, 
\eqref{eq:ball-holder} would imply that
\begin{align*}
\sigma(a,t) 
& < 
(1-3\varepsilon)^{1-a_{j_o}^2} \sqrt{2} \\
& \leq 
(1-3\varepsilon)^{1 - \frac{1}{2(1-\varepsilon)^2}} \sqrt{2} \\
& \leq 
(1-\varepsilon)\sqrt{2}
\end{align*}
where we used that  $a_{j_o} \leq 1/(\sqrt{2}(1-\varepsilon))$ and some algebra. This is a contradiction. Therefore, there exists a second index $j_1 \neq j_o$ such that $I(a_{j_1}) \geq (1-3\varepsilon)\sqrt{2}$.
Proceeding as for $j_o$, we can conclude that necessarily 
$$
a_{j_1} \geq \frac{1}{\sqrt{2}} (1 - \frac{75}{2} \varepsilon) .
$$
The expected result concerning $a_{j_o}$, $a_{j_1}$ follows.

Since $a_1^2+\dots+a_n^2 =1$ we can conclude that 
$$
\sum_{j \neq j_0,j_1} a_j^2 \leq 1 - \frac{1}{2}(1-\frac{25}{2}\varepsilon)^2
- \frac{1}{2}(1-\frac{75}{2}\varepsilon)^2 \leq 50 \varepsilon .
$$
Thus, $a_j^2 \leq 50 \varepsilon$ for all $j \neq j_o,j_1$. 
This ends the proof of the theorem.

\begin{lem}\label{lem:s}
Let $s \geq 2$ be such that
$$
\int_{-\infty}^\infty\left| \frac{\sin(\pi u)}{\pi u} \right|^{s} du 
\geq  
(1-\delta) \sqrt{\frac{2}{s}}
$$
for some small $\delta > 0$.  Then, $s \leq 2 + 50\delta$.
\end{lem}

\begin{proof}
Set $\sigma=\frac{s}{2}-1$. We use the technology developed in \cite{ball}
where it is proved that
$$
\int_{-\infty}^\infty\left| \frac{\sin(\pi u)}{\pi u} \right|^{s} du 
= 
\frac{1}{\pi} \int_{-\infty}^\infty \left| \frac{\sin^2(t)}{t^2} \right|^{1+\sigma} dt 
=
1 - \sum_{n=1}^\infty \frac{|\sigma(\sigma-1)\dots(\sigma-n+1)|}{n!}\beta_n
$$
and
$$
\sqrt{\frac{2}{s}} = \sqrt{\frac{1}{1+\sigma}} 
= 
\frac{1}{\pi} \int_{-\infty}^\infty \left( e^{-t^2/\pi} \right)^{1+\sigma} dt 
=
1 - \sum_{n=1}^\infty \frac{|\sigma(\sigma-1)\dots(\sigma-n+1)|}{n!}\alpha_n
$$
with
$$
\alpha_n \coloneqq \frac{1}{\pi} \int_{-\infty}^\infty e^{-t^2/\pi} \left( 1 - e^{-t^2/\pi} \right)^n dt,
\qquad
\beta_n \coloneqq \frac{1}{\pi} \int_{-\infty}^\infty \frac{\sin^2(t)}{t^2} \left( 1 - \frac{\sin^2(t)}{t^2}\right)^n dt .
$$
Therefore, the assumption
$$
\int_{-\infty}^\infty\left| \frac{\sin(\pi u)}{\pi u} \right|^{s} du 
\geq  
(1-\delta) \sqrt{\frac{2}{s}}
$$
can be recast
$$
\sum_{n=1}^\infty \frac{|\sigma(\sigma-1)\dots(\sigma-n+1)|}{n!}(\beta_n-\alpha_n) \leq \delta \sqrt{\frac{1}{1+\sigma}} . 
$$
Note that, in \cite{ball}, it is proved that $\alpha_n < \beta_n$ so that the left hand side of the latter is positive and in fact an infinite sum of positive terms. Hence, the first term of the sum
must not exceed the right hand side. Since $\beta_1=\frac{1}{3}$ and $\alpha_1=  \frac{\sqrt{2}-1}{\sqrt{2}}$, it holds 
$$
\sigma \frac{3-2\sqrt{2}}{3\sqrt{2}} = \sigma (\beta_1-\alpha_1)  
\leq 
\delta \sqrt{\frac{1}{1+\sigma}} \leq \delta .
$$
Returning to the variable $s$ it follows that $s \leq 2 + \delta \frac{6 \sqrt{2}}{3-2\sqrt{2}}$ from which the expected result follows since $\frac{6 \sqrt{2}}{3-2\sqrt{2}} \simeq 49.46\leq 50$.
\end{proof}


\section{min-Entropy power inequality}
In this section we extend the quantitative slicing results for the unit cube, to a quantitative version (Corollary \ref{Ball to min-EPI} below) of Bobkov and Chistyakov's min-entropy power inequality (Inequality \eqref{eq:bc}) for random variables in $\mathbb{R}$.  Then we prove the full characterization of extremizers of this min-entropy power inequality, \textit{i.e.}\ we prove Theorem \ref{thm:equality_cases}.

The quantitative version  of Bobkov and Chistyakov's min-entropy power inequality reads as follows.

\begin{cor} \label{Ball to min-EPI}
For $X_i$ independent random variables and $\varepsilon \in (0,1/75)$ if 
\begin{align} \label{eq:almost-equality}
    N_\infty \left((1-\varepsilon) \sum_{i=1}^n X_i \right) \leq  \frac{1}{2} \sum_{i=1}^n N_\infty(X_i) ,
\end{align}
then there exists indices $i_o$ and $i_1$ such that
\begin{align*}
   (1-37.5\varepsilon)^2 \left(\frac{1}{2} \sum_{i=1}^n N_\infty(X_i) \right) \leq N_\infty(X_{i_o}), N_\infty(X_{i_1}) \leq (1+2\varepsilon)^2 \left( \frac{1}{2} \sum_{i=1}^n N_\infty(X_i) \right)
\end{align*}
while
\begin{align*}
    \sum_{i \neq i_o, i_1} N_\infty(X_i) \leq 50 \varepsilon \sum_{i=1}^n N_\infty(X_i).
\end{align*}
\end{cor}

Its proof relies on the following result by Rogozin.

\begin{thm}[Rogozin \cite{rogozin1987estimate}] \label{thm: original Rogozin}
For $X_i$ independent random variables, let $Z_i$ be independent random variables uniform on an origin symmetric interval chosen such that $N_\infty(X_i) = N_\infty(Z_i)$, with the interpretation that $Z_i$ is deterministic, and equal to zero, in the case that $N_\infty(X_i) = 0$.  Then,
\begin{align*}
    N_\infty(X_1 + \cdots + X_n) \geq N_\infty(Z_1 + \cdots + Z_n).
\end{align*}
\end{thm}

Note that our frame work here is formally more general than \cite{rogozin1987estimate} and \cite{bobkov2014bounds}

\begin{proof}[Proof of Corollary \ref{Ball to min-EPI}]
Suppose that, for $\delta > 1$
\begin{align} \label{eq:start}
    N_\infty(X_1 + \cdots + X_n) \leq  \frac \delta 2  \sum_{i=1}^n N_\infty(X_i),
\end{align}
then by Theorem \ref{thm: original Rogozin},
\begin{align*}
    N_\infty(Z_1 + \cdots + Z_n) \leq \frac{\delta} 2  \sum_{i=1}^n N_\infty(X_i).
\end{align*}
Writing $U_i = \frac{Z_i}{\sqrt{N_\infty(Z_i)}}$ and $\theta_i = \frac{N_\infty(X_j)}{\sum_j N_\infty(Z_j)}$ we can re-write this inequality as
\begin{align*}
    N_\infty (\theta_1 U_1 + \cdots + \theta_n U_n) \leq \frac{\delta}{2}
\end{align*}
where we observe that $\theta = (\theta_1, \dots, \theta_n)$ is a unit vector and $U = (U_1, \dots , U_n)$ is the uniform distribution on the unit cube.  Moreover since $U_i$ are log-concave and symmetric, $\sum_i \theta_i U_i = \langle \theta, U \rangle$ is as well, and hence $N_\infty(\theta_1 U_1 + \cdots + \theta_n U_n) = f^{-2}_{\langle \theta, U \rangle}(0) = \sigma^{-2}(\theta,0)$.  Thus, we have 
$$
\sigma(\theta,0) \geq \sqrt{\frac{2}{\delta}}.
$$

Now observe that the min-entropy is $2$-homegeneous, \textit{i.e.}\ $N_\infty(\lambda X) = \lambda^2 N_\infty(X)$. Therefore, \eqref{eq:almost-equality} reads as \eqref{eq:start}  with $\delta = (1-\varepsilon)^{-2}$. Hence
\begin{align*}
    \sigma(\theta,0) \geq 
    (1- \varepsilon) \sqrt{2}.
\end{align*}
Thus by Theorem \ref{thm: quantitative Ball}, there exist $i_o$ and $i_1$ such that
\begin{align*}
    \frac{1}{\sqrt{2}} (1-37.5 \varepsilon) \leq \theta_{i_o}, \theta_{i_1} \leq \frac{1}{\sqrt{2}} ( 1 + 2\varepsilon) 
\end{align*}
while
\begin{align*}
    \sum_{i \neq i_o,i_1} \theta_i^2 \leq 50 \varepsilon.
\end{align*}
Interpreting this in terms of the definition $\theta_j = \sqrt{N_\infty(X_j)/\sum_{i} N_\infty(X_i)}$.  This gives,
\begin{align*}
   (1-37.5\varepsilon)^2 \left(\frac{1}{2} \sum_{i=1}^n N_\infty(X_i) \right) \leq N_\infty(X_{i_o}), N_\infty(X_{i_1})) \leq (1+2\varepsilon)^2 \left( \frac{1}{2} \sum_{i=1}^n N_\infty(X_i) \right),
\end{align*}
while,
\begin{align*}
    \sum_{i \neq i_o, i_1} N_\infty(X_i) \leq 50 \varepsilon \sum_{i=1}^n N_\infty(X_i).
\end{align*}
This ends the proof of the Corollary.
\end{proof}

\begin{proof}[Proof of Theorem \ref{thm:equality_cases}]
We distinguish between sufficiency and necessity. The former being simpler.\\
{\it{- Necessity:}}\\
Writing for convenience $N_\infty(X_1)\geq N_\infty(X_2) \geq \cdots \geq N_\infty(X_n)$ by Corollary \ref{Ball to min-EPI} when $N_\infty(X_1) >0$, equality in \eqref{eq: min epi} implies that
\begin{align*}
    N_\infty(X_1) = N_\infty(X_2), \hspace{5mm} N_\infty(X_k) = 0 \hbox{ for } k \geq 3.
\end{align*}
That is
\begin{align*}
    N_\infty(X_1 + X_2 + X_3 \cdots + X_n) = N_\infty(X_1 + X_2) = N_\infty(X_1).
\end{align*}
and since symmetric rearrangement preserves min-entropy and reduces the entropy of independent sums, $N_\infty(X_1 + X_2) \geq N_\infty(X_1^*+ X_2^*) \geq \frac 1 2 (N_\infty(X_1^*) +N(X_2^*)) = N_\infty(X_1) = N_\infty(X_1 + X_2). $  Letting $f,g$ represent the densities of $X_1^*$ and $X_2^*$ respectively, this implies
\begin{align*}
    \|f*g\|_\infty = f*g(0) = \int f(y) g(y) dy = \int_{\{f =\| f\|_\infty \}} \!\!\! \|f\|_\infty g(y) dy + \int_{\{ f  < \|f\|_\infty\}} \!\!\! f(y) g(y) dy  = \|f \|_\infty
\end{align*}
which can only hold if $\{g > 0\} \subseteq \{ f  = \|f \|_\infty\}$.  Reversing the roles of $f$ and $g$, we must also have $\{f > 0 \} \subseteq \{ g = \|g\|_\infty\}$.  Since $\{f = \|f\|_\infty\} \subseteq \{ f > 0\}$ obviously holds, we have the following chain of inclusions,
\begin{align*}
    \{g > 0\} \subseteq \{ f  = \|f \|_\infty\} \subseteq \{ f > 0\} \subseteq \{ g = \|g\|_\infty\}  \subseteq \{ g > 0 \}.
\end{align*}
For this it follows that $X_1^*$ and $X_2^*$ are \textit{i.i.d.}\ uniform distributions.

Thus, $X_1$ and $X_2$ are uniform distributions as well.  Without loss of generality we may assume that $X_1$ and $X_2$ are uniform on sets of measure $1$, $K_1$ and $K_2$. Denote $f_i = \mathbbm{1}_{K_i}$.  Then
$f_1 * f_2$ is uniformly continuous and $f_1*f_2(x) \to 0$ with $|x| \to \infty$.  Indeed, because continuous compactly supported functions are dense in $L^2$, it follows\footnote{Given an $\varepsilon >0$, there exists $\phi$ continuous and compactly supported such that $\|\varphi - g\|_2 < \varepsilon/3$. Since $\varphi$ is continuous and compactly supported, it is uniformly continuous, and hence for small enough $y$, $\|\varphi_{\tau_y} - y\|_2 < \varepsilon/3$, Thus $\| g_{\tau_y} - g\| \leq \| g_{\tau_y} - \varphi_{\tau_y} \| + \| \varphi_{\tau_y} - \varphi\| + \|\varphi - g\| < \varepsilon.$ } that for $g_{\tau_y}(x) \coloneqq g(x+y)$, $\|g_{\tau_y} - g\|_2 \to 0$ for $y \to 0$.  Further $\| g_{\tau_{y_1}} - g_{\tau_{y_2}}\|_2 = \| g_{\tau_{y_1 - y_2}} - g\|_2$, so that for $|y_1 - y_2|$ sufficiently small, $\| g_{\tau_{y_1}} - g_{\tau_{y_2}}\|_2$ can be made arbitrarily small as well.  Thus,
\begin{align*}
    |f_1*f_2(x) - f_1*f_2(x')|
        &\leq
            \int |f_1(-y)||f_2(x+y) - f_2(x'+y)| dy
                \\
        &\leq
            \|f_1\|_2 \|(f_2)_{\tau_x} - (f_2)_{\tau_{x'}} \|_2
                \\
        &=
            \| (f_2)_{\tau_{x-x'}} - f_2 \|_2
\end{align*}
hence $f_1 * f_2$ is indeed uniformly continuous.

Taking $\varphi_i$ to be continuous, compactly supported functions approximating $f_i$ in $L^2$, we have
\begin{align*}
    \| \varphi_1 * \varphi_2 - f_1 * f_2 \|_\infty 
        &\leq 
            \| f_1 * (\varphi_2 - f_2) \|_\infty + \| \varphi_2*(\varphi_1 - f_1) \|_\infty 
                \\
        &\leq 
            \| f_1 \|_2 \|\varphi_2 - f_2\|_2 + \|\varphi_2\|_2 \|\varphi_1 - f_1\|_2 .
\end{align*}
Since the right hand side goes to zero, and $\varphi_1 * \varphi_2$ is compactly supported, it must be true that $f_1 *f_2(x)$ tends to zero for large $|x|$.  Thus $f_1*f_2$  attains its maximum value at some point $x$, and thus we can rewrite the equality of the min-entropies of $X_1+X_2$, $X_1$ and $X_2$, as $f_1*f_2(x) = |K_1 \cap (x - K_2)| = |K_1| = |K_2| =1$.  Thus almost surely $x - K_1 = K_2$.

Put $Y = X_2 + \cdots + X_n$. By the same argument, since $N_\infty (X_1 + Y) = \frac 1 2 \left(N_\infty(X_1) + N_\infty(Y) \right)$, $Y$ is uniform on a set $x' - K_1$.  Thus, $\Var(Y) = \sum_{i=2}^n \Var(X_i) = \Var(X_2)$. Hence, for $i > 2$, $\Var(X_i) = 0$ and the $X_i$ are deterministic.  Letting $A = K_1$, the proof of necessity is complete.

\medskip

\noindent
{\it{- Sufficiency:}}\\
To prove sufficiency, assume that $X_1$ is uniform on a set $A$, $X_2$ uniform on $x - A$ and $X_i$ a point mass for $i \geq 3$ then,
\begin{align*}
    N_\infty(X_1 + X_2 + X_3 + \cdots + X_n)
       & =
            N_\infty(X_1 + X_2) \\
            & = \left \| \frac{\mathbbm{1}_{A}}{|A|}* \frac{\mathbbm{1}_{x-A}}{|A|} \right\|^{-2}_\infty .
\end{align*}
Observe that
\begin{align*}
    \frac{\mathbbm{1}_{A}}{|A|}* \frac{\mathbbm{1}_{x-A}}{|A|}(x)
        &=
            \frac{1}{|A|^2} \int \mathbbm{1}_{A}(y)  \mathbbm{1}_{x-A}(x-y) dy
                \\
        &=
            \frac{1}{|A|},
\end{align*}
Thus $|A|^2 \geq N_\infty(X_1 + X_2)$ and it follows that $|A|^2= N_\infty(X_1 + X_2) = N_\infty(X_1) = N_\infty(X_2)$.
\end{proof}


\section{Quantitative khintchine's inequality}

In this section we prove Theorem \ref{th:quantitative khintchine}
that resembles the proof of Theorem \ref{thm: quantitative Ball}.
We need to recall some results from \cite{haagerup}.

Assume without loss of generality that $a_k \neq 0$ for all $k$.
Put
$$
F(s)=\frac{2}{\pi} \int_0^\infty \left(1 - \left|\cos \left( \frac{t}{\sqrt{s}} \right)\right|^s\right) \frac{dt}{t^2}, \qquad \qquad s >0 . 
$$
From \cite[Lemma 1.4 (and its  proof)]{haagerup}, we can extract that
$$
F(s) 
= 
\frac{2}{\sqrt{\pi s}} \frac{\Gamma \left( \frac{s+1}{2}\right)}{\Gamma \left( \frac{s}{2}\right)}
= 
\sqrt{\frac{2}{\pi}} \prod_{k=0}^\infty \left(1-\frac{1}{(s+2k+1)^2}\right)^\frac{1}{2}
$$ 
is an increasing function of $s$, with $F(2)=1/\sqrt{2}$ and $\lim_{s \to \infty} F(s) = \sqrt{\frac{2}{\pi}}$. 
Haagerup also proved \cite[Lemma 1.3]{haagerup} that
\begin{equation} \label{eq:h1}
R_1(a) \geq \sum_{k=1}^n a_k^2 F \left( \frac{1}{a_k^2}\right)
\end{equation}
with the convention that $a_k^2 F \left( \frac{1}{a_k^2}\right) = 0$ if $a_k=0$ (recall the definition of $R_1$ from \eqref{eq:khintchine}).
For completeness, let us reproduce the argument using Nazarov and Podkorytov's presentation \cite{NP00}. From the identity
$$
|s| = \frac{2}{\pi} \int_0^\infty (1 - \cos(st)) \frac{dt}{t^2}
$$
applied to $s = \sum_{k=1}^n a_k B_k$, we have
\begin{align*}
R_1(a) 
& = 
\mathbb{E} \left( \left| \sum_{k=1}^n a_k B_k \right|\right) \\
& =
\frac{2}{\pi} \int_0^\infty \left(1 - \mathbb{E} \left( \cos \left( t \sum_{k=1}^n a_k B_k \right) \right) \right) \frac{dt}{t^2} \\
& =
\frac{2}{\pi} \int_0^\infty \left(1 - \prod_{k=1}^n \cos(a_k t) \right) \frac{dt}{t^2} \\
\end{align*}
where at the last line we used that
\begin{align*}
\mathbb{E} \left( \cos \left( t \sum_{k=1}^n a_k B_k \right)\right)
& = 
\mathrm{Re} \left( \mathbb{E} \left( e^{i t \sum_{k=1}^n a_k B_k} \right) \right) \\
& =
\prod_{k=1}^n \cos(a_k t) .
\end{align*}
Since $\sum a_k^2=1$, the following Young's inequality $\prod_{k=1}^n s_k^{a_k^2} \leq \sum a_k^2 s_k$ holds  for any non-negative $s_1,\dots,s_n$. Therefore, (take $s_k=|\cos(a_k t)|^{a_k^{-2}}$), it holds
\begin{align*}
R_1(a) 
& \geq 
\frac{2}{\pi} \int_0^\infty \left(1 - \prod_{k=1}^n |\cos(a_k t)| \right) \frac{dt}{t^2} \\
& \geq 
\frac{2}{\pi} \int_0^\infty \left(1 - \sum_{k=1}^n a_k^2 |\cos(a_k t)|^{a_k^{-2}} \right) \frac{dt}{t^2} \\
& = 
\sum_{k=1}^n a_k^2 \frac{2}{\pi} \int_0^\infty \left(1 - |\cos(a_k t)|^{a_k^{-2}} \right) \frac{dt}{t^2} 
\end{align*}
which amounts to \eqref{eq:h1}.


Now observe that $R_1(a) \geq \max_k |a_k|$. Indeed, given $k_o$, multiplying by $B_{k_o}$, that satisfies $|B_{k_o}|=1$, it holds
\begin{align*}
R_1(a)
& = 
\mathbb{E}\left(|B_{k_o}| \left| \sum_{k=1}^n a_k B_k \right| \right) \\\ 
& =
\mathbb{E}\left( \left| a_{k_o} + \sum_{k=1}^n a_k B_kB_{k_o} \right| \right) \\
& \geq \left| \mathbb{E} \left( a_{k_o} + \sum_{k=1}^n a_k B_kB_{k_o} \right) \right| \\
& = |a_{k_o}| .
\end{align*}
It follows by assumption that $|a_{k}| \leq \frac{1+\varepsilon}{\sqrt 2}$ for any $k$.

Assume that $F(1/a_k^2) > \frac{1+\varepsilon}{\sqrt 2}$ for all $k$. Then, by \eqref{eq:h1} and monotonicity of $F$, it would hold
$$
R_1(a) 
\geq 
\sum_{k=1}^n a_k^2 F \left( \frac{1}{a_k^2}\right) 
> 
 \frac{1+\varepsilon}{\sqrt 2} .
$$
This contradicts the starting hypothesis $R_1(a) 
\leq \frac{1+\varepsilon}{\sqrt 2}$. Therefore, there exists at least one index, say $k_o$, such that $F(1/a_{k_o}^2) \leq \frac{1+\varepsilon}{\sqrt 2}$. Using Lemma \ref{lem:treport} we can conclude that
$$
|a_{k_o}| \geq \frac{1}{\sqrt{2}} \frac{1}{\sqrt{1+20 \varepsilon}} \geq \frac{1-10 \varepsilon}{\sqrt{2}}
$$
since $1/\sqrt{1+t} \geq 1 - \frac{t}{2}$ for any $t \in (0,1)$.

We iterate the argument. Assume that $F(1/a_k^2) > \frac{1+3\varepsilon}{\sqrt 2}$ for all $k \neq k_o$. From
\eqref{eq:h1} and monotonicity of $F$, it would hold (recall that $|a_{k}| \leq \frac{1+\varepsilon}{\sqrt 2}$ for any $k$ and in particular for $k_o$)
$$
R_1(a) 
\geq 
\sum_{k=1}^n a_k^2 F \left( \frac{1}{a_k^2}\right) 
> 
 \frac{1+3\varepsilon}{\sqrt 2} \sum_{k \neq k_o} a_k^2 + a_{k_o}^2 F \left( \frac{1}{a_{k_o}^2} \right)
\geq  
\frac{1+3\varepsilon}{\sqrt 2} \sum_{k \neq k_o} a_k^2 + a_{k_o}^2 F \left( \frac{2}{(1+\varepsilon)^2}\right)  .
$$
Now Lemma \ref{lem:treport2} guarantees that $F \left( \frac{2}{(1+\varepsilon)^2}\right) \geq \frac{1-\alpha \varepsilon}{\sqrt{2}}$, with $\alpha=\frac{\pi^2}{12}$, so that, since $\sum_{k \neq k_o}a_k^2 = 1 - a_{k_o}^2$ and
$|a_{k_o}| \leq \frac{1+\varepsilon}{\sqrt 2}$, it holds
\begin{align*}
R_1(a) & 
> 
\frac{1+3\varepsilon}{\sqrt 2} \sum_{k \neq k_o} a_k^2 + a_{k_o}^2\frac{1- \alpha \varepsilon}{\sqrt{2}} \\
& = 
\frac{1+3\varepsilon}{\sqrt 2} + a_{k_o}^2 \left( \frac{1- \alpha \varepsilon}{\sqrt{2}} - \frac{1+3\varepsilon}{\sqrt 2} \right) \\
& \geq 
\frac{1+3\varepsilon}{\sqrt 2} - \left( \frac{1+\varepsilon}{\sqrt{2}}\right)^2\frac{(3+\alpha) \varepsilon}{\sqrt{2}} \\
& = \frac{1+\varepsilon}{\sqrt 2} + \frac{\varepsilon}{4\sqrt{2}}\left(4 - (3+\alpha)(1+\varepsilon)^2 \right) \\
& > \frac{1+\varepsilon}{\sqrt 2}
\end{align*}
since for $\varepsilon  \in (0,1/100)$, $4 > (3+\alpha)(1+\varepsilon)^2$.
This again contradicts the hypothesis $R_1(a) 
\leq \frac{1+\varepsilon}{\sqrt 2}$. Therefore, there exists a second index $k_1 \neq k_o$, such that $F(1/a_{k_1}^2) \leq \frac{1+3\varepsilon}{\sqrt 2}$.
Lemma \ref{lem:treport} then implies that 
$$
|a_{k_1}| \geq \frac{1}{\sqrt{2}} \frac{1}{\sqrt{1+60 \varepsilon}} \geq \frac{1-30 \varepsilon}{\sqrt{2}} 
$$
(since, again, $1/\sqrt{1+t} \geq 1 - \frac{t}{2}$).
This proves the first part of the Theorem.

For the second part we use the previous results together with $\sum_{k=1}^n a_k^2=1$ to ensure that 
$$
\sum_{k \neq k_o, k_1} a_k^2 = 1 - a_{k_o}^2 - a_{k_1}^2 \leq 1 - \left( \frac{1-10 \varepsilon}{\sqrt{2}}\right)^2 - \left( \frac{1-30 \varepsilon}{\sqrt{2}}\right)^2 \leq \frac{80}{\sqrt{2}}\varepsilon \leq 57 \varepsilon .
$$
This ends the proof of the theorem.

\begin{lem} \label{lem:treport}
Fix $\varepsilon \in (0,3/100)$ and $s >0$ such that $F(s) \leq \frac{1+\varepsilon}{\sqrt 2}$. Then
$$
s \leq 2(1+ 20 \varepsilon) .
$$
\end{lem}

\begin{proof}
Assume that $s \geq 2$ (otherwize there is nothing to prove).
By expansion, $F(s) = F(2) + \int_2^s F'(t)dt \leq \frac{1+\varepsilon}{\sqrt 2}$.
Therefore, since $F(2)=1/\sqrt{2}$, 
$$
\int _2^s F'(t)dt \leq \frac{\varepsilon}{\sqrt 2} .
$$
Observe that $F(3) = \frac{4}{\pi \sqrt{3}} \simeq 0.74 \geq 0.71 \simeq \frac{1.01}{\sqrt{2}} \geq \frac{1+\varepsilon}{\sqrt 2}$.  Hence, since $F$ is increasing, necessarily $s \leq 3$. It follows that
$$
(s-2) \inf_{2 \leq t \leq 3} F'(t) \leq \frac{\varepsilon}{\sqrt{2}} 
$$
and we are left with estimating $\inf_{2 \leq t \leq 3} F'(t)$. Using the expression of $F$ above as a product,
we deduce that, for $t \in (2,3)$ 
\begin{align*}
F'(t) 
& = 
F(t) \sum_{k=0}^\infty \frac{1}{(t+2k)(t+2k+1)(t+2k+2)} \\
& \geq 
F(2) \sum_{k=0}^\infty \frac{1}{(2k+3)(2k+4)(2k+5)} \\
& \geq \frac{1}{40 \sqrt{2}}
\end{align*}
where in the last inequality we used that $F(2)=1/\sqrt{2}$ and estimated from below the infinite sum by the first 5 terms\footnote{Alternatively one can argue that
$\sum_{k=0}^\infty \frac{1}{(2k+3)(2k+4)(2k+5)} \geq \sum_{k=0}^\infty \frac{1}{(2k+4)^3} = \frac{1}{8}(\zeta(3)-1)$ where $\zeta(3) \simeq 1.202 \geq 1.2$ is the Riemann zeta function, from which we deduce that the infinite series is bounded below by $1/40$.}. The expected result follows.
\end{proof}

\begin{lem} \label{lem:treport2}
Fix $\varepsilon \in (0,1/100)$. Then
$$
F \left( \frac{2}{(1+\varepsilon)^2}\right)  \geq \frac{1-\alpha \varepsilon}{\sqrt{2}} 
$$
with $\alpha = \pi^2/12$.
\end{lem}

\begin{proof}
By expansion, 
$$
F\left(\frac{2}{(1+\varepsilon)^2} \right) = F(2) + \int_2^{\frac{2}{(1+\varepsilon)^2}} F'(t)dt 
\geq \frac{1}{\sqrt{2}} - \left(2 - \frac{2}{(1+\varepsilon)^2} \right) \sup_{\frac{2}{(1+\varepsilon)^2} \leq t \leq 2}F'(t) .
$$
Now, as in the proof of Lemma \ref{lem:treport}, for any $t \in (\frac{2}{(1+\varepsilon)^2} , 2)$, it holds
$$
F'(t) = F(t) \sum_{k=0}^\infty \frac{1}{(t+2k)(t+2k+1)(t+2k+2)} 
\leq F(2) \sum_{k=0}^\infty \frac{1}{8(k+1)^2} = \frac{\pi^2}{48 \sqrt{2}}
$$
where the inequality follows from the rought estimate
$(t+2k)(t+2k+1)(t+2k+2) \geq 8(k+1)^2$, valid for any $k$ and any $t \in (\frac{2}{(1+\varepsilon)^2} , 2)$ (this is trivial for $t \geq 1$ and $k \geq 1$, the case $k=0$ has to be treated separately, details are left to the reader).

Combining with the previous estimate, we get
$$
F\left(\frac{2}{(1+\varepsilon)^2} \right) 
\geq 
\frac{1}{\sqrt{2}} \left( 1 - \frac{\pi^2}{48}\left(2 - \frac{2}{(1+\varepsilon)^2} \right) \right) 
= 
\frac{1}{\sqrt{2}} \left( 1 - \frac{\pi^2}{24} \frac{2 \varepsilon + \varepsilon^2}{(1+\varepsilon)^2}\right)
\geq 
\frac{1}{\sqrt{2}} \left( 1 - \frac{\pi^2}{12} \varepsilon \right)
$$
which is the expected result.
\end{proof}

\begin{remark}
The range $\varepsilon \in (0,1/100)$ in Theorem \ref{th:quantitative khintchine} is technical and here to guarantee that $(1+\varepsilon)/\sqrt{2} \leq \sqrt{2}/\sqrt{\pi}=\lim_{s \to \infty}F(s)$ and also that 
$F(3) \geq (1+\varepsilon)/\sqrt{2}$ (see the proof of Lemma \ref{lem:treport} above).
\end{remark}


\bibliographystyle{plain}
\bibliography{quantitative-slicing}

\begin{thebibliography}{10}

\bibitem{ball}
K.~Ball.
\newblock Cube slicing in {${\bf R}^n$}.
\newblock {\em Proc. Amer. Math. Soc.}, 97(3):465--473, 1986.

\bibitem{ball2}
K.~Ball.
\newblock Some remarks on the geometry of convex sets.
\newblock In {\em Geometric aspects of functional analysis (1986/87)}, volume
  1317 of {\em Lecture Notes in Math.}, pages 224--231. Springer, Berlin, 1988.

\bibitem{bobkov2014bounds}
S.~G. Bobkov and G.~P. Chistyakov.
\newblock Bounds on the maximum of the density for sums of independent random
  variables.
\newblock {\em Journal of Mathematical Sciences}, 199(2):100--106, 2014.

\bibitem{bobkov2015entropy}
S.~G. Bobkov and G.~P. Chistyakov.
\newblock Entropy power inequality for the {R}{\'e}nyi entropy.
\newblock {\em IEEE Trans. Inf. Theory}, 61(2):708--714, 2015.

\bibitem{bobkov2017variants}
S.~G. Bobkov and A.~Marsiglietti.
\newblock Variants of the entropy power inequality.
\newblock {\em IEEE Transactions on Information Theory}, 63(12):7747--7752,
  2017.

\bibitem{CKT}
G.~Chasapis, H.~K{\"o}nig, and T.~Tkocz.
\newblock From {B}all's cube slicing inequality to {K}hinchin-type inequalities
  for negative moments.
\newblock arXiv preprint arXiv:2011.12251, 2020.

\bibitem{FR}
R.~Frank and H.~Riede.
\newblock Hyperplane sections of the {$n$}-dimensional cube.
\newblock {\em Amer. Math. Monthly}, 119(10):868--872, 2012.

\bibitem{haagerup}
U.~Haagerup.
\newblock The best constants in the {K}hintchine inequality.
\newblock {\em Studia Math.}, 70(3):231--283 (1982), 1981.

\bibitem{hall}
R.~R. Hall.
\newblock On a conjecture of {L}ittlewood.
\newblock {\em Math. Proc. Cambridge Philos. Soc.}, 78(3):443--445, 1975.

\bibitem{hensley}
D.~Hensley.
\newblock Slicing the cube in {${\bf R}^{n}$} and probability (bounds for the
  measure of a central cube slice in {${\bf R}^{n}$} by probability methods).
\newblock {\em Proc. Amer. Math. Soc.}, 73(1):95--100, 1979.

\bibitem{IT}
G.~Ivanov and I.~Tsiutsiurupa.
\newblock On the volume of sections of the cube.
\newblock {\em Anal. Geom. Metr. Spaces}, 9(1):1--18, 2021.

\bibitem{KOS}
R.~Kerman, R.~O\v{l}hava, and S.~Spektor.
\newblock An asymptotically sharp form of {B}all's integral inequality.
\newblock {\em Proc. Amer. Math. Soc.}, 143(9):3839--3846, 2015.

\bibitem{khintchine}
A.~Khintchine.
\newblock \"{U}ber dyadische {B}r\"{u}che.
\newblock {\em Math. Z.}, 18(1):109--116, 1923.

\bibitem{K21}
H.~K\"{o}nig.
\newblock Non-central sections of the simplex, the cross-polytope and the cube.
\newblock {\em Adv. Math.}, 376:107458, 35, 2021.

\bibitem{KK}
H.~K\"{o}nig and A.~Koldobsky.
\newblock On the maximal perimeter of sections of the cube.
\newblock {\em Adv. Math.}, 346:773--804, 2019.

\bibitem{KR}
H.~K\"{o}nig and M.~Rudelson.
\newblock On the volume of non-central sections of a cube.
\newblock {\em Adv. Math.}, 360:106929, 30, 2020.

\bibitem{li2018renyi}
J.~Li.
\newblock R{\'e}nyi entropy power inequality and a reverse.
\newblock {\em Studia Mathematica}, 242:303--319, 2018.

\bibitem{li2020further}
J.~Li, A.~Marsiglietti, and J.~Melbourne.
\newblock Further investigations of {R}{\'e}nyi entropy power inequalities and
  an entropic characterization of s-concave densities.
\newblock In {\em Geometric Aspects of Functional Analysis}, pages 95--123.
  Springer, 2020.

\bibitem{madiman2017forward}
M.~Madiman, J.~Melbourne, and P.~Xu.
\newblock Forward and reverse entropy power inequalities in convex geometry.
\newblock In {\em Convexity and Concentration}, pages 427--485. Springer, 2017.

\bibitem{madiman2017rogozin}
M.~Madiman, J.~Melbourne, and P.~Xu.
\newblock Rogozin's convolution inequality for locally compact groups.
\newblock {\em arXiv preprint arXiv:1705.00642}, 2017.

\bibitem{marsiglietti2018entropy}
A.~Marsiglietti and J.~Melbourne.
\newblock On the entropy power inequality for the {R}{\'e}nyi entropy of order
  [0, 1].
\newblock {\em IEEE Transactions on Information Theory}, 65(3):1387--1396,
  2018.

\bibitem{MR}
J.~Melbourne and C.~Roberto.
\newblock Transport-majorization to analytic and geometric inequalities.
\newblock in preparation, 2021.

\bibitem{NP00}
F.~L. Nazarov and A.~N. Podkorytov.
\newblock Ball, {H}aagerup, and distribution functions.
\newblock In {\em Complex analysis, operators, and related topics}, volume 113
  of {\em Oper. Theory Adv. Appl.}, pages 247--267. Birkh\"auser, Basel, 2000.

\bibitem{ram2016renyi}
E.~Ram and I.~Sason.
\newblock On {R}{\'e}nyi entropy power inequalities.
\newblock {\em IEEE Transactions on Information Theory}, 62(12):6800--6815,
  2016.

\bibitem{renyi1961measures}
A.~R{\'e}nyi et~al.
\newblock On measures of entropy and information.
\newblock In {\em Proceedings of the Fourth Berkeley Symposium on Mathematical
  Statistics and Probability, Volume 1: Contributions to the Theory of
  Statistics}. The Regents of the University of California, 1961.

\bibitem{rioul2018renyi}
O.~Rioul.
\newblock R{\'e}nyi entropy power inequalities via normal transport and
  rotation.
\newblock {\em Entropy}, 20(9):641, 2018.

\bibitem{rogozin1987estimate}
B.~A. Rogozin.
\newblock The estimate of the maximum of the convolution of bounded densities.
\newblock {\em Teoriya Veroyatnostei i ee Primeneniya}, 32(1):53--61, 1987.

\bibitem{shannon1948mathematical}
C.~E. Shannon.
\newblock A mathematical theory of communication.
\newblock {\em The Bell system technical journal}, 27(3):379--423, 1948.

\bibitem{steckin}
S.~B. Ste\v{c}kin.
\newblock On best lacunary systems of functions.
\newblock {\em Izv. Akad. Nauk SSSR Ser. Mat.}, 25:357--366, 1961.

\bibitem{szarek}
S.~J. Szarek.
\newblock On the best constants in the {K}hinchin inequality.
\newblock {\em Studia Math.}, 58(2):197--208, 1976.

\bibitem{young}
R.~M.~G. Young.
\newblock On the best possible constants in the {K}hintchine inequality.
\newblock {\em J. London Math. Soc. (2)}, 14(3):496--504, 1976.

\end{thebibliography}

\end{document}